\documentclass[oneside,12pt,reqno]{compositio}
\usepackage{enumitem} 

\usepackage{amssymb}
\usepackage{amsmath,amsthm,amscd}
\usepackage{verbatim}
\usepackage{graphicx,hyperref}
\usepackage[all]{xy}

\usepackage[usenames,dvipsnames]{xcolor}
\usepackage{amsthm, amsfonts, graphicx, pinlabel, tikz}
\usepackage[all]{xy}
\usepackage{hyperref} 
\usepackage{mathtools}

\usepackage{amssymb}
\usepackage{amsmath,amsthm,amscd}
\usepackage{verbatim}
\usepackage{graphicx,hyperref}

\newcommand{\R}{\mathbb{R}}

\newcommand{\id}{\mathrm{Id}}

\newcommand{\OP}{\operatorname}

\newcommand{\pt}{\operatorname{pt}}

\theoremstyle{plain}
\newtheorem{thm}{Theorem}[section]
\newtheorem{cor}[thm]{Corollary}
\newtheorem{lem}[thm]{Lemma}

\newtheorem{mainthm}{Theorem}

\theoremstyle{definition}

\theoremstyle{remark}
\newtheorem{rem}[thm]{Remark}

\numberwithin{equation}{section}



\title[$C^0$-Legendrians and positive loops]{$C^0$-limits of Legendrians and positive loops}

\author{Georgios Dimitroglou Rizell}
\email{georgios.dimitroglou@math.uu.se}
\address{Department of Mathematics\\
Uppsala University\\
Box 480\\
SE-751 06 Uppsala\\
Sweden}

\author{Michael G. Sullivan}
\email{mikesullivan@umass.edu}
\address{Department of Mathematics and Statistics\\
University of Massachusetts\\
Amherst\\
MA 01003\\
USA}

\thanks{The first author is supported by the Knut and Alice Wallenberg Foundation under the grants KAW 2021.0191 and KAW 2021.0300, and by the Swedish Research Council under the grant number 2020-04426. The second author is supported by the Simons Foundation grant number 708337.
The authors thank Lukas Nakamura for helpful comments and conversations related to Remarks \ref{rem:relax2} and \ref{rem:NonLegDisp}, as well as the anonymous referee for providing significant feedback in their thorough review.}

\begin{document}

\begin{abstract}
We show that the image of a properly embedded  Legendrian submanifold under a homeomorphism that is the $C^0$-limit of a sequence of contactomorphisms supported in some fixed compact subset  is again Legendrian, if the image of the submanifold is smooth.
In proving this, we show that any  closed  non-Legendrian submanifold of a contact manifold admits a positive loop and we provide a parametric refinement of the Rosen--Zhang result on the degeneracy of the Chekanov--Hofer--Shelukhin pseudo-norm for  properly embedded  non-Legendrians.
\end{abstract}

\maketitle

\section{Terminology and notation}

Let $(M, \xi)$ be a $(2n+1)$-dimensional, possibly non-compact, contact manifold with contact distribution $\xi \subset TM.$  We will assume $M$ to be co-orientable and so we can chose a contact 1-form $\alpha$ where $\xi = \ker\{\alpha\}.$ 
We will denote by $\Lambda \subset M$  a  connected properly embedded (not necessarily closed) 
 {\bf Legendrian} (submanifold), which means $\dim(\Lambda) = n$ and $T\Lambda \subset \xi.$ 
We will denote by $K \subset M$  a connected properly embedded smooth submanifold with $\dim(K)  \le n.$ Usually $K$ will be a {\bf non-Legendrian} connected properly embedded submanifold, which means 
either $\dim(K)  < n,$ or $\dim(K) = n$ and there
 there exists $x \in K$ such that $T_xK \not \subset  \xi_x.$ 
 We will sometimes consider these non-Legendrians (and Legendrians) as parameterized; i.e., $K$ (and $\Lambda$) is equipped with an embedding into $M.$ One canonical example is the inclusion $\id_M|_K: K \rightarrow K \subset M.$ Also for any contactomorphism $\Phi \in \OP{Cont}(M,\xi)$ such that $\Phi(K) = K$ and $\Phi|_K \ne \id_M|_K,$ $\Phi|_K$ is another parameterization of $K.$

Recall that the contact isotopies $\Phi^t \colon M \to M$ of a co-orientable contact manifold are in bijective correspondence with the time-dependent smooth functions, so-called contact Hamiltonians $H_t \colon M \to \R$. Note that this bijection depends on the choice of contact one-form $\alpha$. Given a contact isotopy we can recover the contact Hamiltonian by
$$\alpha(\dot\Phi^t)=H_t \circ \Phi^t.$$
Conversely, a contact Hamiltonian determines an isotopy via the equation
$$\alpha(\dot\Phi^t)=H_t \circ \Phi^t, \:\: d \alpha(\dot\Phi^t, \cdot)|_{\ker \alpha}= 
 -dH_t|_{\ker \alpha}$$
for each $t$. We say that $H_t$ {\bf generates} the contact isotopy $\Phi^t$. The contact isotopy generated by $H_t \equiv 1$ is called the {\bf Reeb flow of $\alpha$}.

Note that the \emph{sign} of the contact Hamiltonian at a point only depends on the choice of co-orientation. We say that a (parameterized) submanifold $A \subset (M,\xi)$ {\bf admits a positive loop} if there exists a contact isotopy $\Phi^t$ for which $\Phi^1(A)=A$ (resp.~$\Phi^1 \circ \id_A =\id_A$) for which the generating contact Hamiltonian $H_t$ satisfies $H_t(x) > 0$ for all $x \in \Phi^t(A)$ and $t$. We similarly say that $A$ admits a {\bf somewhere positive non-negative loop} if $H_t(x) \ge 0$ for all $x \in \Phi^t(A)$ and $t$, where the latter inequality moreover is strict for some point $x$ and time $t$. Note that a contact Hamiltonian that vanishes along a  properly embedded  Legendrian submanifold induces a flow that fixes that Legendrian. Furthermore, any somewhere positive non-negative loop of a  closed   Legendrian can be generated by a contact Hamiltonian that is non-negative on the entire ambient manifold $M$.

 In this article all contactomorphisms, homeomorphisms, and isotopies are implicitly assumed to have support contained inside some fixed compact subset,  even though the ambient contact manifold $M$ and the connected properly embedded submanifolds $\Lambda,K$  sometimes  need not be compact.

\section{Statements of results}

In this section, we describe a number of results contrasting flexibility versus rigidity, for Legendrians (loose or not) and non-Legendrians: $C^0$-limits, positive loops, and pseudo-metrics. 

\subsection{$C^0$-topology}

We start with our main result, which is about Legendrians under homeomorphisms that are $C^0$-limits of contactomorphisms.
\begin{mainthm}
\label{thm:C0}
      Consider a sequence $$\Phi_k \colon (M,\xi) \xrightarrow{\cong} (M,\xi)$$ of contactomorphisms supported in a fixed compact set,  and let $\Lambda \subset M$ be a  properly embedded  Legendrian. If $\Phi_k \xrightarrow{C^0} \Phi_\infty$ where $\Phi_\infty$ is a homeomorphism, and if $\Phi_\infty(\Lambda)$ is smooth, then $\Phi_\infty(\Lambda)$ is also Legendrian.    \end{mainthm}

Partial results of Theorem \ref{thm:C0} have appeared elsewhere: Nakamura assumed that there was a uniform lower bound on the length of Reeb chords, as well as some small technical conditions \cite[Theorem 3.4]{Nakamura:C0}; Rosen and Zhang assumed $C^0$-convergence for the smooth  $f_k: M \rightarrow \R$ defined by $\Phi_k^* \alpha = e^{f_k} \alpha$ (often called conformal factors) \cite[Theorem 1.4]{RosenZhang}; Usher relaxed Rosen and Zhang's hypothesis to certain lower bounds on the conformal factors \cite[Theorem 1.2]{Usher:C0}; we proved the general case in dimension 3 \cite[Theorem D]{DRS3}; and Stoki\'{c} made no assumptions in \cite[Proposition 6.1]{Stokic}, but concluded the limiting submanifold could not be nearly Reeb-invariant \cite[Definition 1.3]{Stokic}. Stoki\'{c} showed that not being nearly Reeb-invariant implies being Legendrian in the case when $\dim(\Lambda)=1$ (in higher dimensions we do not know if the analogous result is true). Some of these results assumed the Legendrians were compact.

In \cite[Theorem D]{DRS3} we proved that $\Lambda$ and $\Phi_\infty(\Lambda)$ are contactomorphic Legendrians when $\dim(\Lambda) = 1.$ This equivalence, and many other weaker connections, are still unknown for $\dim(\Lambda) > 1.$  For example, if $\Phi_\infty(\Lambda)$ is loose, then need $\Lambda$ be loose as well? 
We only know the following.
\begin{thm}
  \label{thm:squeezing}
Consider the set-up as in Theorem \ref{thm:C0}.  Assume $\Lambda$ is closed, $\dim(\Lambda) >1$ and  $k \gg 0$.
  \begin{enumerate}
  \item There exists a standard one-jet neighborhood $\Lambda \subset U$ such that the Legendrian  $\Phi_\infty(\Lambda)$  is contained in the one-jet neighborhood $\Phi_k(U).$ Furthermore, inside the one-jet neighborhood $\Phi_k(U)$, the Legendrian $\Phi_k(\Lambda)$ can be squeezed  into a one-jet neighborhood of  $\Phi_\infty(\Lambda)$ and  $\Phi_\infty(\Lambda)$ can be squeezed into a one-jet neighborhood of $\Phi_k(\Lambda), $ in the sense of \cite[ Section 1.2]{DRS20}.
  \item 
  \begin{enumerate}  
  \item
    If $\Lambda$ and $\Phi_\infty(\Lambda)$ are diffeomorphic, then  $\Phi_k(\Lambda)$ and $\Phi_\infty(\Lambda)$  are smoothly isotopic  inside $\Phi_k(U)$.
     \item $\Phi_\infty(\Lambda)$ is not loose inside the one-jet neighborhood $\Phi_k(U).$ 
    \end{enumerate}
  \item  Suppose  some stabilization   $(M \times T^*X, \ker\{\alpha +pdq\})$ of $M$   admits an open contact embedding  into $J^1 \R^N$ and   the  Legendrian  $\Lambda \times {\mathbf{0}_{T^*X}}$  admits a Chekanov--Eliashberg DGA-augmentation as a Legendrian in $J^1 \R^N,$ then $\Phi_\infty(\Lambda)$ is not loose in $M$. (See \cite[Section 1]{DRS20} for a review of Chekanov--Eliashberg differential graded algebras and augmentations in this context.)
  \end{enumerate}
\end{thm}

  \begin{rem}
\label{rem:relax1}  
The assumption in Theorems \ref{thm:C0} and \ref{thm:squeezing} that $\Phi_\infty$ is a homeomorphism can be dropped if we instead assume $\Phi_\infty(U)$ is a neighborhood of $\Phi_\infty(\Lambda)$ for some standard Legendrian one-jet neighborhood $\Lambda \subset U$. See  the proof of Theorem \ref{thm:C0}  and then  use of Theorem \ref{thm:squeezing}.1 in the rest of the proof of Theorem \ref{thm:squeezing}.


\end{rem}

\subsection{Positive loops}

The strategy of the proof of Theorem \ref{thm:C0} is inspired by Stoki\'{c}'s proof of \cite[Proposition 6.1]{Stokic}. However, instead of producing a Reeb-invariant neighborhoods of arbitrary non-Legendrians (it is unclear if they always exist), we show that non-Legendrians admit positive loops in Theorem \ref{thm:NonLegPosLoop}. We then allude to the classical theorem of non-existence of $C^0$-small positive loops of Legendrians proven by Colin--Ferrand--Pushkar \cite{Colin:Positive} (see Theorem \ref{thm:NoPositive}).

In order to produce small positive loops of any non-Legendrian, we first prove the following flexibility when it comes to the choice of contact Hamiltonian for a contact isotopy of any  non-Legendrian submanifold.

 \begin{mainthm}
    \label{thm:ZeroHamiltonian}
    Let $K \subset (M^{2n+1},\xi)$ be   a properly embedded  non-Legendrian  and $\Phi^t \colon M \to M$ be a  contact isotopy. There exists a  contact isotopy $\Psi^t \colon M \to M$  such that the following hold.
    \begin{itemize}
                \item $\Psi^t$ is generated by a contact Hamiltonian $H_t$ that vanishes when restricted to $\Psi^t(K)$;
               \item $\Psi^t(K)$ is contained inside an $\epsilon$-neighborhood of $\Phi^t(K)$ for an arbitrary choice of $\epsilon>0$ and all $t \in [0,1]$; and
            \item $\Psi^1$ and $\Phi^1$ agree in a small neighborhood of $K.$
  \end{itemize}
  \end{mainthm}
In contrast, a contact Hamiltonian that vanishes along a properly embedded  Legendrian submanifold generates a contact isotopy that fixes the Legendrian set-wise.
  
  We continue by establishing some consequences of Theorem \ref{thm:ZeroHamiltonian}, starting with the existence of positive loops.

By Theorem \ref{thm:ZeroHamiltonian} there is a flexibility in the choice of contact Hamiltonian for a contact isotopy of a non-Legendrian submanifold; namely, it shows that we can $C^0$-deform the isotopy to one whose generating contact Hamiltonian vanishes along the image of the submanifold. The following direct consequence shows that there also is flexibility for the behavior of contact isotopies that are  positive  along the Legendrian.

 \begin{mainthm}
    \label{thm:NonLegPosLoop}

Any closed non-Legendrian $K \subset (M^{2n+1},\xi)$ sits in a positive loop.
Equivalently, there exists a contact isotopy  $\Phi^t$ with 
generating Hamiltonian $H_t$ such that the following hold.
     \begin{itemize}
       \item 
         $\Phi^0|_V=\Phi^1|_V=\id_V$
         is satisfied in some neighborhood $V$ of  $K.$
       \item $H_t$  is  positive on the image of $K$ under $\Phi^t$ for all $t \in [0,1].$
         \end{itemize}
       \end{mainthm}
 

       \begin{rem}
          With the ideas behind Theorems \ref{thm:ZeroHamiltonian} and \ref{thm:NonLegPosLoop}, it should be possible to prove that any contact isotopy of a  closed  non-Legendrian can be $C^0$-perturbed to an isotopy for which the contact Hamiltonian is $C^0$-close to any arbitrary function.
        
         \end{rem}

Consider a closed Legendrian $\Lambda \subset M$ that is loose in the sense of Murphy \cite{LooseLeg}.  Since a non-Legendrian push-off $K$ of $\Lambda$ admits a contractible positive loop, and since $\Lambda$ can be squashed  onto $K$ by a contact isotopy (which automatically preserves the positivity) in the sense of \cite{DRS4}, 
 Theorem \ref{thm:NonLegPosLoop} gives a new proof of the following result by Liu:
\begin{cor}[\cite{Liu:Positive}]
\label{cor:Liu}
    Any closed loose Legendrian admits a contractible positive loop. 
  \end{cor}

 This flexibility of non-Legendrians and loose Legendrians contrasts to the rigidity of certain non-loose  Legendrians. 
Colin, Ferrand and Pushkar used generating functions to prove  the non-existence of positive loops for the zero-section in a 1-jet space $0_N \subset (J^1N,\xi_{st})$ of a closed manifold $N$ in \cite[Theorem 1]{Colin:Positive}; in the case of $N=S^n$ the result was obtained independently by Chernov and Nemirovski in  \cite{Cher_Nem}. The latter authors generalized the result to non-negative isotopies of $0_N \subset (J^1N,\xi_{st})$:
\begin{thm}[Corollary 5.5 in \cite{ChernovNemirovski}]
\label{thm:NoPositive}
 The zero-section in the one-jet space $0_N \subset (J^1N,\xi_{st})$ of a (not necessarily closed) manifold $N$ does not admit a somewhere positive non-negative loop supported in a compact subset.
\end{thm}
The original sources assume $N$ is compact, but their arguments apply to our set-up, since the loop is assumed to have compact support.
 Take a double of a large pre-compact open $X \subset N$ with smooth boundary such that $J^1(X)$ contains the support of the purported loop. There is an induced non-negative loop of the zero-section in the jet-space $J^1(X \sqcup_{\partial X} X)$ of the double of $X$, i.e.~the manifold obtained by gluing $X$ to itself along its boundary.

\subsection{The Chekanov--Hofer--Shelukhin pseudo-norm}

Fix a  properly embedded  submanifold $K$ and consider its orbit space under the action of $\OP{Cont}_0(M,\xi),$  the identity component of the space of contactomorphisms. 
The (un-parameterized) {\bf Chekanov--Hofer--Shelukhin pseudo-metric} on this orbit space is defined via
  $$\delta^{\mbox{\tiny{un-p}}}_\alpha(K_0,K_1) \coloneqq \inf\{ \| \Phi^1 \|_\alpha;\:\: \Phi^t \in \OP{Cont}_0(M,\xi), \:\: \Phi^1(K_0)=K_1 \}$$
  where
  $$\| \Phi^1 \|_\alpha=\inf_{\Phi^1_H=\Phi^1}\int_0^1\max_{x\in M}|H_t(x)|dt$$
  is the Shelukhin--Hofer norm on $\OP{Cont}_0(M,\xi)$ \cite{Shelukhin:Hofer}.
  
 We define the {\bf parametrized Chekanov--Hofer--Shelukhin pseudo-metric} on the orbit space of parameterized  embeddings $\phi_i \colon K \hookrightarrow M$ by
  $$\delta_\alpha(\phi_0,\phi_1) \coloneqq \inf\{ \| \Phi^1 \|_\alpha;\:\: \Phi^1 \in \OP{Cont}_0(M,\xi), \:\:  \Phi^1 \circ \phi_0 =\phi_1 \}. $$
Given any parametrized submanifold $\phi \colon K \hookrightarrow M$, we get an induced pseudo-metric on $\OP{Cont}_0(M,\xi)$ by setting
$$ \delta_{\alpha,\phi}(\Phi_0,\Phi_1) \coloneqq \delta_\alpha(\Phi_0 \circ \phi,\Phi_1 \circ \phi).$$
We typically consider this pseudo-metric defined by a choice of submanifold $K \subset M$ with the canonical parametrization $\phi = \id_M|_K \colon K \hookrightarrow M.$

Rosen and Zhang showed that the unparameterized pseudo-metric $\delta^{\mbox{\tiny{un-p}}}_\alpha$ identically vanishes on the orbit space of any closed non-Legendrian submanifold \cite[Theorem 1.10]{RosenZhang}. (This was independently proved later in Nakamura's MS thesis \cite[Corollary D.16]{NakamuraMastersThesis}.) In contrast, $\delta^{\mbox{\tiny{un-p}}}_\alpha$ is non-degenerate on the orbit space of any Legendrian submanifold. The non-degeneracy of $\delta^{\mbox{\tiny{un-p}}}_\alpha$ for Legendrians was first proved by Usher \cite[Corollary 3.5]{Usher:C0} when there are no contractible Reeb orbits or relatively contractible Reeb chords, and then by Hedicke \cite[Theorem 5.2]{Hedicke}  when the Legendrian does not sit in a positive loop. We then proved the non-degeneracy of $\delta^{\mbox{\tiny{un-p}}}_\alpha$ for arbitrary closed Legendrians of closed contact manifolds in \cite[Theorem 1.5]{DRS3}.
The analogous results for the parameterized Chekanov-Hofer-Shelukhin pseudo-metric $\delta_\alpha$  follow readily from Theorem \ref{thm:ZeroHamiltonian}. 
 
\begin{cor}
\label{cor:ParameterizedCHS}
The parameterized Chekanov-Hofer-Shelukhin pseudo-metric $\delta_\alpha$ vanishes identically for any non-Legendrian 
$K.$
For any Legendrian, $\delta_\alpha$ is degenerate.
\end{cor}

\begin{rem}
\label{rem:relax2}
When the Legendrian is closed, \cite[Theorem 1.5]{DRS3} implies that the degenerate $\delta_\alpha$ does not vanish identically.
 But to apply \cite[Theorem 1.5]{DRS3}, we need the contact manifold $M$ to either be closed, or to have a codimension-0 contact embedding into a closed contact manifold $\tilde{M}.$ In this latter case, we moreover require the contact form $\alpha$ of $M$ to be a restriction of a contact form $\tilde{\alpha}$ for $\tilde{M}.$ 
\end{rem}

\section{Proofs of results}

\subsection{Basic results for contact Hamiltonians}
\label{ssec:Basic}

  We start with some preliminary standard computations for contact Hamiltonians that will be useful. In the following we fix a contact form $\alpha$ on $M$ for the correspondence between contact Hamiltonians and contact isotopies.
  \begin{lem}
    \label{lem:inversehamiltonian}
  If $\Phi_i^t \colon M \to M$, $i=0,1,$ are contact isotopes generated by time-dependent contact Hamiltonians $H^i_t \colon M \to \R$ then $\Phi_0^t \circ \Phi^t_1$ is a contact isotopy that is generated by
  $$G_t=H^0_t + e^{f_t \circ (\Phi_0^t)^{-1}}H^1_t \circ (\Phi_0^t)^{-1}$$
  where  the time-dependent function $f_t \colon M \to \R$ is determined by $(\Phi_0^t)^*\alpha=e^{f_t} \alpha$.  In particular, $(\Phi_1^t)^{-1}$ is generated by $-e^{f_t \circ  \Phi_1^t}H^1_t\circ  \Phi_1^t$ (which can be seen by setting $\Phi_0^t \coloneqq (\Phi_1^t)^{-1}$).
\end{lem}
Here and throughout, composition occurs at each $t.$ For example, $\Phi_0^t \circ \Phi^t_1$ is an isotopy with the same time-parameter as $\Phi_0^t$ and $\Phi^t_1.$
\begin{proof}The chain-rule implies
\begin{eqnarray*} 
G_t(\Phi^t_0 \circ \Phi^t_1) & = & 
      \alpha\left(\frac{d}{dt} \left(\Phi^t_0 \circ \Phi^t_1\right)\right) = \alpha\left(\left(\frac{d}{dt}\Phi^t_0\right)\left(\Phi^t_1\right)+D\Phi^t_0 \circ\left(\frac{d}{dt}\Phi^t_1\right)\right) \\
     & =  & H_t^{ 0 }(\Phi^t_0 \circ \Phi^t_1) + e^{f_t \circ \Phi^t_1}H^1_t \circ \Phi^t_1.
\end{eqnarray*}
\end{proof}

\begin{lem}
  \label{lem:backwards}
    If $\Phi^t \colon M \to M$ is a contact isotopy generated by a time-dependent contact Hamiltonian $H_t \colon M \to \R$, then $(\Phi^1)^{-1}\circ \Phi^{1-t}$  is a contact isotopy generated by 
    $G_t=-e^{f\circ \Phi^1}  H_{1-t} \circ \Phi^{1}$  
    where the smooth function $f \colon M \to \R$ is determined by $((\Phi^1)^{-1})^*\alpha=e^f \alpha$.
    In particular, if $H_t$ vanishes along the image of $K$ under $\Phi^t$, then $G_t$ vanishes along the image of  $K$  under  $(\Phi^1)^{-1}\circ \Phi^{1-t}.$ 
  \end{lem}

\begin{proof}
\begin{eqnarray*}
G_t((\Phi^1)^{-1}\circ \Phi^{1-t}) 
     &=&  
      \alpha\left(\frac{d}{dt} \left((\Phi^1)^{-1}\circ \Phi^{1-t}\right)\right) 
     = \alpha\left(D(\Phi^1)^{-1}\left(\frac{d}{dt}\Phi^{1-t}\right)\right)\\
     &=&   -e^{f\circ \Phi^{1-t}} H_{1-t}(\Phi^{1-t}). 
\end{eqnarray*}
\end{proof}

  \begin{lem}
  	\label{lem:conjugation}
	 Let $\Psi \in \OP{Cont}(M,\xi)$ be a contactomorphism not necessarily contact isotopic to the identity. 
    If $\Phi^t \colon M \to M$ is a contact isotopy generated by a time-dependent contact Hamiltonian $H_t \colon M \to \R$, then $\Psi \circ \Phi^t \circ \Psi^{-1}$ is a contact isotopy generated by
         $ G_t=e^{f\circ \Psi^{-1}} H_t \circ \Psi^{-1}$
where  
    $\Psi^*\alpha=e^f \alpha$.  
    In particular, if $H_t$ vanishes along the image of $K$ under $\Phi^t$, then $G_t$ vanishes along the image of $\Psi(K)$ under $\Psi \circ \Phi^t \circ \Psi^{-1}$.
  \end{lem}

\begin{proof}
\begin{eqnarray*}
G_t(\Psi \circ \Phi^t \circ \Psi^{-1}) 
     &=& 
\alpha\left(\frac{d}{dt} \left(\Psi \circ \Phi^t \circ \Psi^{-1}\right)\right)
 = \alpha\left(D\Psi\left(\frac{d}{dt}\Phi^t\right)\left(\Psi^{-1}\right)\right)\\
 &=& e^{f \circ \Phi^t \circ \Psi^{-1}}H_{t}\left(\Phi^t \circ \Psi^{-1}\right).
\end{eqnarray*}
  
\end{proof}

\subsection{Proof of Theorem \ref{thm:ZeroHamiltonian}}

Usher proved that the ``rigid locus'' of a half-dimensional non-Lagrangian submanifold of a symplectic manifold is empty in  \cite[Corollary 2.7]{Usher:Observations}. 
This was later generalized to the contact setting by Rosen--Zhang in \cite{RosenZhang} and independently by Nakamura \cite{Nakamura:C0}. This means, in particular, that the unparameterized Hofer--Chekanov--Shelukhin pseudo-norm vanishes when restricted to non-Legendrians, i.e.~that two  contact isotopic  non-Legendrians are contact isotopic via contact Hamiltonians of arbitrarily small norm. Our strategy here is to translate the proofs in the aforementioned works to yield a more direct construction of the deformed contact isotopy generated by a small contact Hamiltonian. This leads to Theorem~\ref{thm:ZeroHamiltonian}, which sharpens the result from \cite{RosenZhang} in the following two ways.
\begin{itemize} 
\item The deformed contact isotopy can be assumed to be generated by a contact Hamiltonian that \emph{vanishes} along the image of the non-Legendrian (as opposed to just being arbitrarily small there).
  \item The time-one map of the deformed contact isotopy can be assumed to induce the same parametrization as the original one, when restricted to the non-Legendrian.
  \end{itemize}
The latter property can be rephrased as saying that the parametrized version of the Hofer--Chekanov--Shelukhin pseudo-norm vanishes when restricted to non-Legendrian submanifolds; see Corollary \ref{cor:ParameterizedCHS}.

We first simplify the problem to the case when $\Phi^t$ is  $C^\infty$-small.

\begin{lem}
\label{lem:LocalTheoremB}
Fix  $\epsilon >0.$ 
If Theorem \ref{thm:ZeroHamiltonian} holds for any contact isotopy whose $C^\infty$-norm is bounded by $\epsilon,$ then Theorem \ref{thm:ZeroHamiltonian} holds for any contact isotopy.
\end{lem}

\begin{proof}
This follows since (a finite number of) concatenated isotopies preserve the three properties of Theorem \ref{thm:ZeroHamiltonian}.
Note that concatenation here is not a composition of maps at each time $t,$ as it was in Section \ref{ssec:Basic}.
\end{proof}

By Banyaga's  fragmentation result \cite[p.148]{Banyaga:Structure} (see also Rybicki \cite{Rybicki}),
the concatenation preservation of the three properties of Theorem \ref{thm:ZeroHamiltonian} enables us to assume, when proving Theorem \ref{thm:ZeroHamiltonian}, 
that $\Phi^t$ is not only $C^\infty$-small, but also supported in a small neighborhood of some $\pt \in K$.
(If the small support of $\Phi^t$ does not intersect  $K,$ the proof is trivial as we set $\Psi^t := \id$ in Theorem \ref{thm:ZeroHamiltonian}.) 


  \begin{lem}
    \label{lem:DisplacementTrick}
 Consider the given non-Legendrian $K,$ contact isotopy $\Phi^t,$ and constant $\epsilon$ from Theorem \ref{thm:ZeroHamiltonian}. 
Further, assume that $\Phi^t$ is supported in a neighborhood $U \subset M$ that is  displaced from a neighborhood $V \subset M$ of the non-Legendrian  $K \subset M$ by a contact isotopy $\widetilde{\Phi}^t$. If 
$\OP{supp} \widetilde{\Phi}^t \supset U$ is contained inside an $\epsilon$-neighborhood of $K$
 and if $\widetilde{\Phi}^t$ is generated by a Hamiltonian which vanishes on $\widetilde{\Phi}^t(K),$ 
 then Theorem \ref{thm:ZeroHamiltonian} holds for this  $K, \Phi^t, \epsilon.$
  \end{lem}

  \begin{proof}
  	By Lemma \ref{lem:backwards} the contact isotopy  $(\widetilde{\Phi}^1)^{-1} \circ \widetilde{\Phi}^{1-t}$  satisfies the property that its generating Hamiltonian vanishes along the image of $K$ under the isotopy. Moreover, the time-one map of this contact isotopy displaces  $V$ from $U$ (because $\widetilde{\Phi}^1(U) \cap V = \emptyset$ implies  $(\widetilde{\Phi}^1)^{-1} \circ \widetilde{\Phi}^{1-1}(V)  \cap U=\emptyset$). 
  	
    The contact isotopy  $\Psi^t$  is constructed by concatenating (locally) the contact isotopy    $(\widetilde{\Phi}^1)^{-1} \circ \widetilde{\Phi}^{1-t}$  that displaces $V$  from $U$, with the  contact isotopy 
$$ \Phi^1 \circ (\widetilde{\Phi}^1)^{-1} \circ \widetilde{\Phi}^t \circ \widetilde{\Phi}^1  \circ (\Phi^1)^{-1} = 
\left(\Phi^1 \circ (\widetilde{\Phi}^1)^{-1} \right) \circ \widetilde{\Phi}^t \circ \left(\Phi^1 \circ (\widetilde{\Phi}^1)^{-1}\right)^{-1},$$
which is the conjugation of the isotopy $\widetilde{\Phi}^t$ with a contactomorphism $\Phi^1 \circ (\widetilde{\Phi}^1)^{-1}$  that might not be equal to the identity inside $U \supset \OP{supp} \Phi^1$ at $t=0$.
(We concatenate these two paths as in the proof of Lemma \ref{lem:LocalTheoremB}.  Technically, $\Psi^t$ is defined for $0 \le t \le 2,$ but to simplify notation, we omit this needed reparameterization of $t.$)

 Set the $\Psi$ and $\Phi^t,$ as used in the notation of Lemma \ref{lem:conjugation}, equal to $\Phi^1 \circ (\widetilde{\Phi}^1)^{-1}$ and $\widetilde{\Phi}^t$ as used to define the second contact isotopy in the preceding paragraph.
 Lemma \ref{lem:conjugation}  implies that this second contact isotopy is generated by a Hamiltonian which vanishes on the image of  $\Phi^1 \circ (\widetilde{\Phi}^1)^{-1}(K) = (\widetilde{\Phi}^1)^{-1} (K)$ under this second isotopy. 
 For this last equality, recall that $\Phi^1|_{(\widetilde{\Phi}^1)^{-1} (K)} = \id|_{(\widetilde{\Phi}^1)^{-1} (K)}$ because  $\Phi^t$ is supported in $U$ which does not intersect $(\widetilde{\Phi}^1)^{-1} (V) \supset (\widetilde{\Phi}^1)^{-1} (K).$

      
Finally,  we conclude that 
$$\Psi^1|_V = \Phi^1 \circ (\widetilde{\Phi}^1)^{-1}\circ\widetilde{\Phi}^1\circ\widetilde{\Phi}^1\circ 
 (\Phi^1)^{-1} \circ (\widetilde{\Phi}^1)^{-1} \circ
\widetilde{\Phi}^0  |_V  =\Phi^1 |_V $$
 where in the last equality we use  $(\Phi^1)^{-1}|_{(\widetilde{\Phi}^1)^{-1} (V)} = \id|_{(\widetilde{\Phi}^1)^{-1} (V)}$ and $\widetilde{\Phi}^0 = \id$. 
    \end{proof}


    
\begin{lem}
      \label{lem:displace1}
 Consider $\epsilon>0$ and $K$ from Theorem \ref{thm:ZeroHamiltonian}.
Fix $p \in K.$ 
There is a neighborhood $U \subset M$ of $p,$  a neighborhood $V \subset M$ of  $K$  and a contact isotopy $\widetilde{\Phi}^t$ which displaces $U$ from $V \cup U$, such that $\widetilde{\Phi}^t$ satisfies its assumptions in Lemma \ref{lem:DisplacementTrick} (i.e.~$\OP{supp}\widetilde{\Phi}^t \supset U$ is contained inside an $\epsilon$-neighborhood of $K$  and $\widetilde{\Phi}^t$ is generated by a Hamiltonian which vanishes on $\widetilde{\Phi}^t(K)$ ). 
\end{lem}
\begin{proof}
Given any choice of neighborhood of $K$, the below construction can be carried out inside that neighborhood. This implies the sought property of the support of the contact isotopy that we now proceed to define.

{\bf 1. The case when $T_pK$ is not a Lagrangian subspace of $\xi_p$:} 

This part of the argument is similar to \cite[Proposition 8.6]{RosenZhang}.

The property that $T_pK$ is not a Lagrangian subspace is equivalent to 
$$(T_pK \cap \xi_p)^{d\alpha} \neq T_pK \cap \xi_p,$$
 where $(T_pK \cap \xi_p)^{d\alpha} \subset \xi_p$ denotes the symplectic orthogonal  (recall that $\dim T_pK \le n$). Hence, we can find a non-zero vector
$$X_H \in (T_pK \cap \xi_p)^{d\alpha} \setminus T_pK \subset \xi_p \setminus T_pK$$
such that the one-form  $\eta\coloneqq d\alpha( \cdot, X_H )$ on $T_pM$  vanishes on $T_pK$. The one-form $\eta$ extends to the exterior derivative $dH$ of a function $H \colon M \to \R$ that can be taken to vanish on all of $K$.

Consider the contact isotopy $\widetilde{\Psi}^t$ generated by the autonomous contact Hamiltonian $H$.  Since $H$ vanishes on $p$ we get $\dot{\Psi}^0(p)=X_H.$ In view of Lemma \ref{lem:inversehamiltonian}, the inverse  $\widetilde{\Phi}^t\coloneqq (\widetilde{\Psi}^t)^{-1}$ is generated by the non-autonomous contact Hamiltonian  $G_t \coloneqq  -e^{f_t \circ (\widetilde{\Phi}^t)^{-1}}H\circ (\widetilde{\Phi}^t)^{-1}.$ In particular, $\widetilde{\Phi}^t$ is generated by a contact Hamiltonian that vanishes along the image of $K$ under $\widetilde{\Phi}^t.$

Finally, since the contact vector field $-X_H=X_{G_0}$ is normal to $K$ at $p$, it follows that $G_t$ generates a contact isotopy that displaces a small neighborhood $p \in U \subset M$ from  $V \cup U$ for some small neighborhood $V$ of $K.$ 

{\bf 2. The case when $T_pK \subset \xi_p$ is a Lagrangian  subspace:} 

Consider the closed subset $\mathcal{L}(K) \subset K$ of points for which $T_pK \subset \xi$ is Lagrangian (which of course is empty whenever $\dim K < n$).

{\bf 2.1 The case when $p \in \OP{bd}\mathcal{L}(K)$:}

First, by a construction which is similar to the one above, for any $X \in T_pK \subset \xi$ 
we can construct a time-dependent contact Hamiltonian $H$ that vanishes on the image of $K$ under the generated isotopy $\Phi^t_X,$ and for which the corresponding contact vector field at time $t=0$ satisfies $X_{H}(p)=X.$ Since $X$ is tangent to $K$, it is not necessarily the case that $p$ is displaced by $\Phi^t_X$ for small $t \ge 0$.

If we can find some $X$ such that $\Phi^\epsilon_X(p) \notin K$  for all small $\epsilon>0$ then we are done. Assume not, then  $\Phi_X^{ \epsilon }(p) \in K$  for all $X$ and  for some $ \epsilon \ge 0.$  Since we are in the case $p \in \OP{bd}\mathcal{L}(K)$, the point $p$ does not have a Legendrian neighborhood in $K$. We can thus find a  direction  $X$ for which  $\Phi_X^{ \epsilon }(p)$  is contained inside $K \setminus \mathcal{L}(K)$. Note that,  for this reason, $K$ and  $\Phi^\epsilon_X(K)$  are not tangent at  $\Phi^\epsilon_X(p).$ Take a non-zero  tangent vector
$W \in T_{\Phi_X^\epsilon(p)}\Phi^\epsilon_X(K) \subset \xi_{\Phi_X^\epsilon(p)}$
that is normal to $K$. By the assumption above, we can find a contact isotopy of $\Phi^\epsilon_X(K)$ whose infinitesimal generator is equal to $W$ at $\Phi^\epsilon_X(p)$, and such that the generating contact Hamiltonian vanishes along the image of $\Phi^\epsilon_X(K)$ under the isotopy.

The latter contact isotopy displaces $\Phi^\epsilon_X(p)$ from some small neighborhood $V$ of $K,$ 
and the concatenation of contact isotopies thus displaces a small neighborhood $U$ of $p$ from $V$  as sought. 


{\bf 2.2 The case when $p \in \OP{int}\mathcal{L}(K)$:}

Finally,  since $K$ is connected,  any point in the  open Legendrian submanifold $\mathcal{L}(K) \setminus \OP{bd}\mathcal{L}(K)$ can be moved arbitrarily close to a point $p' \in  \OP{bd}\mathcal{L}(K)$ by a contact isotopy that fixes $K \setminus (\mathcal{L}(K) \setminus \OP{bd}\mathcal{L}(K))$ point-wise and $\mathcal{L}(K) \setminus \OP{bd}\mathcal{L}(K)$ set-wise. 
Note that  the Hamiltonian of such a contact isotopy can be taken to vanish on all of $K.$
(See \cite[Section 1]{DRS20} and the proof of \cite[Theorem 2.6.2]{Geiges2008}.)
We then apply the contact isotopy from case {\bf{2.1}} to $p'$ displacing its small neighborhood $U'$ from $V.$ This also displaces a smaller neighborhood $U \subset U'$ of $p$ from $V$ as sought. 

\end{proof}

\begin{rem}
\label{rem:NonLegDisp}
Global infinitesimal displaceability of non-Legendrians is an important ingredient in Nakamura's work \cite{Nakamura:C0}. The vanishing of the unparameterized Chekanov--Hofer--Shelukhin norm \cite{RosenZhang} implies that any closed non-Legendrian has a displacement that can be realized by a contact Hamiltonian that is arbitrarily $C^0$-small. To that end we use the fact that, for any non-Legendrian $K$, a generic Reeb vector field is nowhere tangent to $K$. Hence, the Reeb flow is a contact isotopy that displaces the non-Legendrian $K$. 
\end{rem}

\subsection{Proof of Corollary  \ref{cor:ParameterizedCHS}}

The parameterized pseudo-metric $\delta_\alpha$ is clearly degenerate on Legendrian submanifolds, since reparametrizing a Legendrian can be done by a contact Hamiltonian that vanishes along the Legendrian. See \cite[Section 1]{DRS20} and the proof of \cite[Theorem 2.6.2]{Geiges2008}. 
It is non-vanishing because $\delta^{\mbox{\tiny{un-p}}}_\alpha$ is non-degenerate \cite[Theorem 1.5]{DRS3}.
The pseudo-metric $\delta_\alpha$ vanishes identically for any non-Legendrian because of the first and third bullet points of Theorem \ref{thm:ZeroHamiltonian}.

\subsection{Proof of Theorem \ref{thm:NonLegPosLoop}}

%

Let $\rho:  M  \rightarrow [-1,0]$ be a smooth compactly supported bump function such that $\rho|_{ U } = -1$  for some ``large" (see below) neighborhood $U \supset K.$   
Apply Theorem \ref{thm:ZeroHamiltonian}, setting $K$  and $\Phi^t$ in Theorem \ref{thm:ZeroHamiltonian}
to be   $K$  in Theorem \ref{thm:NonLegPosLoop} and the  flow induced by the autonomous contact Hamiltonian $\epsilon\rho$, respectively. Note that this flow is equal to the negative Reeb flow on  $U$  rescaled by $\epsilon>0$, which is assumed to be small.
Theorem \ref{thm:ZeroHamiltonian} produces $\Psi^t$
by constructing $\Psi^t$ to be local to the image of the compact $K$ under the negative Reeb flow. So without loss of generality, we can assume $U$ is sufficiently large such that $\OP{supp}(\Psi^t) \subset U.$ 
We claim that $(\Phi^t)^{-1} \circ \Psi^t$ is the desired isotopy in Theorem \ref{thm:NonLegPosLoop} (which unfortunately is also called $\Phi^t$ in that theorem).

That $(\Phi^t)^{-1} \circ \Psi^t$ satisfies the first bullet point of Theorem \ref{thm:NonLegPosLoop} follows from the third bullet point of Theorem \ref{thm:ZeroHamiltonian}.

Note that $\Phi^t$ is generated by a Hamiltonian which is negative on $U$, and  thus  negative on the support of $\Psi^t.$  The second part of Lemma \ref{lem:inversehamiltonian}, setting $\Phi_1^t = \Phi^t,$ implies that $(\Phi^t)^{-1}$ is generated by a  Hamiltonian which is positive on $U,$ and  thus  positive on the support of $\Psi^t.$ The first part of Lemma \ref{lem:inversehamiltonian}, setting $\Phi_0^t = (\Phi^t)^{-1}, \Phi_1^t = \Psi^t,$ combined with the first bullet point of Theorem \ref{thm:ZeroHamiltonian} applied to $\Psi^t,$ implies that $(\Phi^t)^{-1} \circ \Psi^t$ satisfies the second bullet point of Theorem \ref{thm:NonLegPosLoop}.
(To see this, using the notation of Lemma \ref{lem:inversehamiltonian}, the Hamiltonian $G_t$ which generates $\Phi_0^t\circ \Phi_1^t,$ when restricted to  $\Phi_0^t \circ\Phi_1^t(K),$ is a sum of the positive term $H^0_t|_{\Phi_0^t\circ\Phi_1^t(K)}$ and $e^{f_t \circ (\Phi^t_0)^{-1}} H^1_t \circ (\Phi^t_0)^{-1}|_{\Phi_0^t\circ \Phi_1^t(K)}.$
But the second term vanishes because $H^1_t|_{\Phi_1^t(K)} =0.$)

\subsection{Proof of Theorem \ref{thm:C0}}

First we show that the general case when $K$ and $\Lambda$ are properly embedded, but not necessarily closed, can be deduced from the statment in the case when the involved submanifolds are assumed to be closed.

Recall that the sequence $\Phi_k$ of contactomorphisms all are assumed to have support inside some fixed compact subset. Take a compact domain $U \subset M$ with smooth boundary that contains the support of all contactomorphisms $\Phi_k$ in the sequence, which in particular means that $\Lambda=K$ holds in a neighborhood of $\overline{M \setminus U}$. For a generic choice of domain $U$, we may further assume that the intersection
$$B \coloneqq \Lambda \cap \partial U=K \cap \partial U$$
is transverse, yielding a smooth submanifold $B \subset \Lambda$ of codimension one. After deforming the neighborhood $U$ near $B$ we may further assume that there is a neighborhood $O$ of $\Lambda \cap U$ in $U$ that is contactomorphic to $J^1(\Lambda \cap U)$, under which $j^10$ is identified with $\Lambda \cap U$. 

Now produce an open contact manifold $\tilde{M}$ from $\OP{int} U$ in the following manner. Let $\tilde{\Lambda}$ denote the closed manifold obtained by gluing two disjoint copies of $\Lambda \cap U$ along its common bondary in the obvious manner. Clearly $J^1\tilde{\Lambda}$ contains $J^1(\Lambda \cap U)$ as a properly embedded submanifold with boundary. Hence we can glue $J^1\tilde{\Lambda}$ to $\OP{int} U$, resulting in an open contact manifold $\tilde{M}$ in which $\Lambda$ extends to a closed Legendrian $\tilde{\Lambda}$. Applying the statement to the closed Legendrian manifold $\tilde{\Lambda} \subset \tilde{M}$ we deduce that
$$\tilde{K} \coloneqq (K \cap U) \cup(\tilde{\Lambda} \setminus \OP{int}U)$$ is Legendrian, and hence so is the original submanifold $K$.

It remains to prove Theorem \ref{thm:C0} when $\Lambda$ and $K$ are closed, which we prove by contradiction. Suppose $K \subset M$ is a closed non-Legendrian submanifold that is the $C^0$-limit of the closed Legendrian submanifolds $\Phi_n(\Lambda)$ where $\Phi_n$ are contactomorphisms that $C^0$-converge to a homeomorphism: $\Phi_n\xrightarrow{C^0} \Phi_\infty.$
 
Let $\Phi^t$ be the contact isotopy from Theorem \ref{thm:NonLegPosLoop} generated by $H_t$ such  that $H_t|_{\Phi^t(K)}>0.$
Take a sufficiently small neighborhood $U \supset K$ contained inside $V$ provided by the theorem, so that $H_t|_{\Phi^t(U)}>0$ as well as $\Phi^1|_U=\id$ are satisfied. Since $\Phi_k(\Lambda) \subset U$ holds for all $k \gg 0$, we have produced a positive loop $\Phi^t|_{\Phi_k(\Lambda)}$ of a closed Legendrian submanifold.  This  contradicts Theorem \ref{thm:NoPositive}. Hence, $K$ is Legendrian. \qed

%
%

\subsection{Proof of Theorem \ref{thm:squeezing}}

  \begin{enumerate}
  \item Since $\Phi_\infty$ is a homeomorphism, $\Phi_\infty(U)$ is a neighborhood of $\Phi_\infty(\Lambda).$ 
  (Or see Remark \ref{rem:relax1}.) The $C^0$-convergence implies that $\Phi_k(U)$ is a neighborhood of $\Phi_\infty(\Lambda)$ for $k \gg 0$. 
 The  $C^0$-convergence  ensures that the fiber-wise rescaling inside $\Phi_k(U)$ projects $\Phi_\infty(\Lambda)$ onto $\Phi_k(\Lambda)$  with  degree 1, as required in this squeezing \cite[ Section 1.2 ]{DRS20}. Since $\Phi_k(\Lambda)$ also may be assumed to be contained inside a one-jet neighborhood of $\Phi_\infty(\Lambda)$, then there is also a  squeezing of $\Phi_k(\Lambda)$ into this one-jet neighborhood in the same sense (i.e.~the fiber-wise projection is of degree one). 
  \item
 \begin{enumerate}
 \item
 Recall that two maps from the same domain that are sufficiently $C^0$-close are homotopic; $\Phi_k|_\Lambda$ is thus homotopic to $\Phi_\infty|_\Lambda$ inside $\Phi_k(U)$ for $k \gg 0.$  Diffeomorphic and homotopic implies smoothly isotopic in high dimensions \cite{Plongements}. 
  \item
To show that $\Phi_\infty(\Lambda)$ is not loose in the one-jet neighborhood $\Phi_k(\Lambda),$  we claim that the zero-section in $J^1\Lambda$ cannot be squeezed  into the one-jet neighborhood of  a loose Legendrian, while (1) provides such a squeezing. To see that the zero-section cannot be squeezed into the one-jet neighborhood of  a loose Legendrian we argue as follows. After stabilizing the ambient contact manifold $(M,\alpha)$ to $(M \times T^*S^1,\alpha + p\,dq)$, the Legendrian $\Lambda$ to $\Lambda \times {\mathbf{0}_{T^*S^1}}$, and the contactomorphisms $\Phi_k$ to $\Phi_k \times \id_{T^*S^1}$, we may consider the case when $U \cong J^1 (\Lambda \times S^1)$ since we can stabilize the squeezing.
The  result follows from Lemma \ref{lem:displacement}.
 \end{enumerate}
    \item This follows directly from \cite[Theorem 1.7]{DRS20}.   
     (In \cite{DRS20}, the term ``stabilized" Legendrian is used in a completely different sense than $\Lambda \times  {\mathbf{0}_{T^*X}}$ as above. When $\dim(\Lambda)=1,$ in a local front projection, a neighborhood of a point of $\Lambda$ is replaced by a zig-zag. We use this zig-zag construction in the proof of Lemma \ref{lem:displacement} below.
 When $\dim(\Lambda) >1,$ stabilization is defined by a more general construction which Murphy proves  equivalent to the existence of a loose chart \cite{LooseLeg}.) 
    \end{enumerate}
    \qed
    
    \begin{lem}
      \label{lem:displacement}
    The zero-section in $J^1(\Lambda \times S^1)$ cannot be squeezed into the one-jet neighborhood of a loose Legendrian submanifold.
    \end{lem}    
    
    \begin{proof}
 Let $\Pi: J^1 \Lambda \rightarrow   T^* \Lambda$ be the projection along the standard Reeb flow.
A Legendrian $\Lambda' \subset J^1 \Lambda$ is horizontally displaceable if there exists a contact isotopy that disjoins $\Lambda'$ from its image  under this  Reeb flow, $\Pi^{-1}(\Pi(\Lambda')).$ 
This is an open condition in the sense that if $\Lambda'$ is horizontally displaceable, then so too is any Legendrian that is sufficiently $C^0$-close to $\Lambda'.$

The Rabinowitz Floer complex (for example see \cite[Section 4]{DRS3}) of the zero-section $j^1(0)$ of $J^1\Lambda $ is not acyclic and its homology is invariant under contact isotopy.
Since the complex is generated by Reeb chords  between $j^1(0)$ and its image under the contact isotopy, this implies $j^1(0)$  is not horizontally displaceable.

      Let $Z_{\OP{st}} \subset J^1S^1$ denote the stabilized zero-section in $J^1S^1$ whose front is given by a single zig-zag.
This Legendrian is horizontally displaced when after the contact isotopy, the minimum magnitude of the slope of the zig-zag is greater than its initial zig-zag slope's maximum magnitude. 
 It follows that the stabilized zero-section $j^1(0) \times Z_{\OP{st}} \subset J^1(\Lambda \times S^1)$ also admits a horizontal displacement.

We claim that by  Murphy's h-principle \cite{LooseLeg}, any compact  
loose Legendrian $ \Lambda_0  \subset J^1(\Lambda \times S^1)$ can be placed   inside a one-jet neighborhood of $j^1(0) \times Z_{\OP{st}}$ by a contact isotopy.
 To see the isotopy, we construct a formal Legendrian isotopy between the loose Legendrian $\Lambda_0$ and a \emph{formal} Legendrian $\Lambda_1$ contained in the neighborhood of $j^1(0) \times Z_{st}$. The formal Legendrian isotopy is constructed by, first, fiber scaling $\Lambda_0$ towards the zero-section. Then, since $Z_{st} \subset J^1S^1$ is smoothly isotopic to $j^10$, it is now easy to see that there is a smooth isotopy $f_t:L \rightarrow J^1(\Lambda \times S^1)$ such that $f_i(L) =  \Lambda_i$ for $i=0,1$, where the smooth (not necessarily Legendrian) submanifold $\Lambda_1$ is contained inside the one-jet neighborhood of $j^10 \times Z_{st}$.  The formal isotopy $(g,G): [0,1]_t \times [0,1]_s \times L \rightarrow \left(J^1(\Lambda \times S^1), T(J^1(\Lambda \times S^1)) \right)$ is defined as follows: set $g_{t,s} = f_t,$ $G_{t,0} = df_t,$
$G_{0,s} = df_0;$  and then use homotopy-lifting to extend $G_{t,s}$ as a full-rank bundle map whose image is a Lagrangian in the contact planes for all $t\in[0,1]$ and $s=1$. The $0$-parametric version of Murphy's h-principle produces an actual loose Legendrian $\Lambda_1'$ in an arbitrarily small neighborhood of $\Lambda_1$. Since $\Lambda_1'$ is formally Legendrian isotopic to $\Lambda_0$ by construction, the $1$-parametric version of Murphy's h-principle produces the Legendrian isotopy that takes $\Lambda_0$ into the neighborhood, as sought.

Hence the loose Legendrians are all horizontally displaceable as well.
If  the zero-section of $J^1(\Lambda \times S^1)$ can be squeezed  into the one-jet neighborhood of some loose Legendrian, then by fiber-scaling it can be squeezed into an arbitrarily small one-jet neighborhood of the loose Legendrian. So the zero-section  is horizontally displaceable, contradicting its Rabinowitz Floer calculation.

      \end{proof}

\bibliographystyle{alpha}
\bibliography{references}

\end{document}